\documentclass{article}
\usepackage{amsfonts}
\usepackage{amsmath}
\usepackage{amssymb}
\usepackage{geometry}
\usepackage{setspace}
\usepackage{sectsty}
\usepackage{lineno}
\usepackage{theorem}
\usepackage{fancyhdr}
\usepackage{graphicx}

\newtheorem{theorem}{Theorem}

\newtheorem{condition}[theorem]{Condition}

\newtheorem{corollary}[theorem]{Corollary}

\newtheorem{lemma}[theorem]{Lemma}

\newtheorem{proposition}[theorem]{Proposition}

\newenvironment{proof}[1][Proof]{\noindent\textbf{#1.} }{\ \rule{0.5em}{0.5em}}
\DeclareMathOperator{\Span}{Span}

\input{tcilatex}

\begin{document}

\title{A concentration inequality for the excess risk in least-squares
regression with random design and heteroscedastic noise}
\author{Adrien Saumard \\
%EndAName
Crest-Ensai, Universit\'{e} Bretagne Loire}
\maketitle

\begin{abstract}
We prove a new and general concentration inequality for the excess risk in
least-squares regression with random design and heteroscedastic noise. No
specific structure is required on the model, except the existence of a
suitable function that controls the local suprema of the empirical process.
So far, only the case of linear contrast estimation was tackled in the
literature with this level of generality on the model. We solve here the
case of a quadratic contrast, by separating the behavior of a linearized
empirical process and the empirical process driven by the squares of
functions of models.

\textit{keywords}: regression, least-squares, excess risk, empirical
process, concentration inequality, margin relation.

\textit{AMS2000}: 62G08, 62J02, 60E15.
\end{abstract}

\section{Introduction}

The excess risk of a M-estimator is a fundamental quantity of the theory of
statistical learning. Consequently, a general theory of rates of convergence
as been developed in the nineties and early 2000 (%
%TCIMACRO{%
%\TeXButton{\cite{Massart:07, Koltchinskii:11}}{\cite{Massart:07, Koltchinskii:11}}}%
%BeginExpansion
\cite{Massart:07, Koltchinskii:11}%
%EndExpansion
). However, it has been recently identified that some theoretical
descriptions of learning procedures need finer controls than those brought
by the classical upper bounds of the excess risk. In this case, the
derivation of \textit{concentration inequalities for the excess risk} is a
new and exiting axis of research, of particular importance for obtaining
satisfying oracle inequalities in various contexts, especially linked to
high dimension.

In the field of model selection, it has been indeed remarked that such
concentration inequalities allow to discuss the non-asymptotic optimality of
model selection procedures (%
%TCIMACRO{%
%\TeXButton{\cite{BirMas:07, ArlotMassart:09, Ler:11, saum:12}}{\cite{BirMas:07, ArlotMassart:09, Ler:11, saum:12}}}%
%BeginExpansion
\cite{BirMas:07, ArlotMassart:09, Ler:11, saum:12}%
%EndExpansion
). More precisely, concentration inequalities for the excess risk and for
the excess empirical risk (\cite{BouMas:10}) are central tools to access the
optimal constants in the oracle inequalities describing the model selection
accuracy. Such results have put to evidence the optimality of the so-called
slope heuristics (%
%TCIMACRO{%
%\TeXButton{\cite{BirMas:07, BauMauMich:12}}{\cite{BirMas:07, BauMauMich:12}}}%
%BeginExpansion
\cite{BirMas:07, BauMauMich:12}%
%EndExpansion
) and more generally of selection procedures based on the estimation of the
minimal penalty (\cite{Arl_Bac:2009}), in statistical frameworks linked to
regularized quadratic estimation. Under similar assumptions, it is also
possible to discuss optimality of resampling and cross-validation type
procedures (%
%TCIMACRO{%
%\TeXButton{\cite{Arl:2008a, Arl_Ler:2012:penVF:JMLR, celisse2014, navarro2015slope}}{\cite{Arl:2008a, Arl_Ler:2012:penVF:JMLR, celisse2014, navarro2015slope}}}%
%BeginExpansion
\cite{Arl:2008a, Arl_Ler:2012:penVF:JMLR, celisse2014, navarro2015slope}%
%EndExpansion
).

In high dimension, convex methods allow to design and compute efficient
estimators. This is the reason why Chatterjee \cite{chatterjee2014} has
recently focused on the estimation of the mean of a high dimensional
Gaussian vector under convex constraints. By getting a concentration
inequality for the excess risk of a projected least-squares estimator,
Chatterjee \cite{chatterjee2014} has proved the universal admissibility of
this estimator. The concentration inequality has then been sharpened and
extended to the excess risk of estimators minimizing penalized convex
criteria (%
%TCIMACRO{%
%\TeXButton{\cite{MurovandeGeer:15, vandeGeerWain:16}}{\cite{MurovandeGeer:15, vandeGeerWain:16}}}%
%BeginExpansion
\cite{MurovandeGeer:15, vandeGeerWain:16}%
%EndExpansion
).

It is also well known (see for instance \cite{bellec2017towards}) that a
weakness of the theory of regularized estimators in a sparsity context is
that classical oracle inequalities, such as in \cite{BickRitovTsy:09},
describe the performance of estimators with an amount of regularization that
actually depends on the confidence level considered in the oracle
inequality. This does not correspond to any practice with this kind of
estimators, whose regularization parameter is usually fixed using a
cross-validation procedure. Recently, Bellec and Tsybakov \cite%
{bellec2016bounds}, building on \cite{bellec2016slope}, have established
more satisfying oracle inequalities, describing the performance of
regularized estimators such as LASSO, group LASSO and SLOPE with a
confidence level independent of the regularization parameter. In particular,
the oracle inequalities can be integrated. Again, the central tool to obtain
such bound is a concentration inequality for the excess risk estimator at
hand.

In this paper, we extend the technology developed in 
%TCIMACRO{%
%\TeXButton{\cite{chatterjee2014, MurovandeGeer:15, vandeGeerWain:16}}{\cite{chatterjee2014, MurovandeGeer:15, vandeGeerWain:16}} }%
%BeginExpansion
\cite{chatterjee2014, MurovandeGeer:15, vandeGeerWain:16}
%EndExpansion
in order to establish a new concentration inequality for the excess risk in
least-squares regression with random design and heteroscedastic noise. This
is appealing since \cite{chatterjee2014} and \cite{MurovandeGeer:15} cover
only least-squares regression with fixed design and homoscedastic Gaussian
noise, while \cite{vandeGeerWain:16} have to assume that the law of the
design is known in order to perform a "linearized least-squares regression"
(see Section 6.2 of \cite{vandeGeerWain:16}).

Our strategy is as follows. We first remark that the empirical process of
interest splits into two parts: a linear process and a quadratic one. Then
we prove that the linear process achieves a second order margin condition as
defined in \cite{vandeGeerWain:16} and put meaningful conditions on the
quadratic process in order to handle it. Techniques from empirical process
theory such as Talagrand's type concentration inequalities and contraction
arguments are at the core of our approach.

The paper is organized as follows. The regression framework as well as some
properties linked to margin relations are described in Section \ref%
{section_framework}. Then we state our main result in Section \ref%
{section_proofs}. The proofs are deferred to Section \ref{section_proofs}.

\section{Least-squares heteroscedastic regression with random design\label%
{section_framework}}

\subsection{Setting\label{ssection_setting}}

Let $\left\{ \left( X_{i},Y_{i}\right) \right\} _{i=1}^{n}$ be an i.i.d.
sample taking values in $\mathcal{X\times \mathbb{R}},$ where $\mathcal{X}$
is a measurable space - typically a subset of $\mathbb{R}^{p}$. We assume
that the following relation holds%
\begin{equation*}
Y_{i}=g_{\ast }\left( X_{i}\right) +\sigma \left( X_{i}\right) \varepsilon
_{i}\text{ \ \ \ \ for }i=1,...,n\text{ },
\end{equation*}%
where $g_{\ast }$ is the regression function, $\sigma $ is the
heteroscedastic noise level and $\mathbb{E}\left[ \varepsilon _{i}\left\vert
X_{i}\right. \right] =0$, $\mathbb{E}\left[ \varepsilon _{i}^{2}\left\vert
X_{i}\right. \right] =1$.

We take a closed convex model $\mathcal{G}\subset L_{2}\left( P^{X}\right) $%
, where $P^{X}$ is the common distribution of the $X_{i}^{\prime }s$, and
set 
\begin{equation*}
g^{0}=\arg \min_{g\in \mathcal{G}}\left\{ P\left( \gamma \left( g\right)
\right) \right\} \text{ ,}
\end{equation*}%
where $P$ is the common distribution of the pairs $\left( X_{i},Y_{i}\right) 
$ and $\gamma $ is the least-squares contrast, defined by%
\begin{equation*}
\gamma \left( g\right) \left( x,y\right) =\left( y-g\left( x\right) \right)
^{2}\text{ .}
\end{equation*}%
We will also denote $f_{g}:=\gamma \left( g\right) ,$ $f_{\ast }=\gamma
\left( g_{\ast }\right) $, $f^{0}=\gamma \left( g^{0}\right) ,$ $\hat{f}%
=\gamma \left( \hat{g}\right) $ and0 $\mathcal{F}:=\gamma \left( \mathcal{G}%
\right) $. The function $g^{0}$ will be called the \textit{projection} of
the regression function onto the model $\mathcal{G}$. Indeed, if we denote $%
\left\Vert \cdot \right\Vert $ the quadratic norm in $L_{2}\left(
P^{X}\right) $, it holds%
\begin{equation*}
\left\Vert g^{0}-g_{\ast }\right\Vert =\min_{g\in \mathcal{G}}\left\Vert
g-g_{\ast }\right\Vert \text{ .}
\end{equation*}%
We consider the least-squares estimator $\hat{g}$ over $\mathcal{G}$,
defined to be%
\begin{equation*}
\hat{g}\in \arg \min_{g\in \mathcal{G}}\left\{ P_{n}\left( \gamma \left(
g\right) \right) \right\} \text{ ,}
\end{equation*}%
where $P_{n}:=1/n\sum_{i=1}^{n}\delta _{\left( X_{i},Y_{i}\right) }$ is the
empirical measure associated to the sample.

We want to assess the concentration of the quantity $P\left( \hat{f}%
-f^{0}\right) =P\left( \gamma \left( \hat{g}\right) -\gamma \left(
g^{0}\right) \right) $, called the \textit{excess risk} of the least-squares
estimator on $\mathcal{G}$, around a single deterministic point. To this
end, it is easy to see (\cite{saum:12}, Remark 1, \cite{navarro2015slope},
Section 5.2 or van de Geer and Wainwright \cite{vandeGeerWain:16}) that the
following representation formula holds for the excess risk on $\mathcal{G}$
in terms of empirical process,%
\begin{equation}
\hat{s}:=\sqrt{P\left( \hat{f}-f^{0}\right) }=\arg \min_{s\geq 0}\left\{
s^{2}-\mathbf{\hat{E}}_{n}\left( s\right) \right\} \text{ ,}
\label{def_s_hat}
\end{equation}%
where 
\begin{equation*}
\mathbf{\hat{E}}_{n}\left( s\right) =\max_{f\in \mathcal{F}_{s}}\left\{
\left( P_{n}-P\right) \left( f^{0}-f\right) \right\} \text{ ,}
\end{equation*}%
with $\mathcal{F}_{s}=\left\{ f\in \mathcal{F};P\left( f-f^{0}\right) \leq
s^{2}\right\} $. It is shown in \cite{vandeGeerWain:16} for various settings
that include linearized regression that the quantity $\hat{s}$ actually
concentrates around the following point,%
\begin{equation}
s_{0}:=\arg \min_{s\geq 0}\left\{ s^{2}-\mathbf{E}\left( s\right) \right\} 
\text{, \ \ }\mathbf{E}\left( s\right) :=\mathbb{E}\left[ \mathbf{\hat{E}}%
_{n}\left( s\right) \right] \text{ .}  \label{def_s0}
\end{equation}

\subsection{On a margin-like relation pointed on the projection of the
regression function\label{ssection_margin_like}}

In order to prove concentration inequalities for the excess risk on $%
\mathcal{G}$, we will need to check the following margin-like relation, also
called "quadratic curvature condition" in \cite{vandeGeerWain:16}: there
exists a constant $C>0$ such that%
\begin{equation}
P\left( f-f^{0}\right) \geq \frac{\sigma ^{2}\left( f-f^{0}\right) }{C^{2}}%
\text{ for all }f\in \mathcal{F}\text{ ,}  \label{quad_curv}
\end{equation}%
where%
\begin{equation*}
\sigma ^{2}\left( f\right) :=\mathbb{E}f^{2}\left( X_{1}\right) -\left( 
\mathbb{E}f\left( X_{1}\right) \right) ^{2}
\end{equation*}%
is the variance of $f$ with respect to $P^{X}$.

A very classical relation in statistical learning, called margin relation,
consists in assuming that (\ref{quad_curv}) holds with $f^{0}$ is replaced
by the image of the target $f_{\ast }=\gamma \left( g_{\ast }\right) $. Such
relation is satisfied in least-squares regression whenever the response
variable $Y$ is uniformly bounded. Here we do not assume that $f_{\ast }$
belongs to $\mathcal{F}$ thus $f^{0}$ may be different from $f_{\ast }$.

\begin{condition}
\label{cond_reg_bounded}There exists $A_{1}>0$ such that $\left\vert
Y\right\vert \leq A_{1}$ $a.s.$
\end{condition}

From Condition \ref{cond_reg_bounded} we deduce that $\left\Vert g_{\ast
}\right\Vert _{\infty }\leq A_{1}$ and $\left\Vert \sigma \right\Vert
_{\infty }\leq 2A_{1}$.

\begin{condition}
\label{cond_model_bounded}There exists $A_{2}>0$ such that%
\begin{equation*}
\sup_{g\in \mathcal{G}}\left\Vert g\right\Vert _{\infty }\leq A_{2}<\infty 
\text{ .}
\end{equation*}
\end{condition}

From Conditions \ref{cond_reg_bounded} and \ref{cond_model_bounded}, we
deduce that the image model $\mathcal{F}$ is also uniformly bounded: there
exists $K>0$ such that 
\begin{equation*}
\sup_{f\in \mathcal{F}}\left\Vert f-f^{0}\right\Vert _{\infty }\leq K<\infty 
\text{ .}
\end{equation*}%
More precisely, $K=2\left( A_{1}+A_{2}\right) $ is convenient.

The following proposition shows that relation (\ref{quad_curv}) is satisfied
in our regression setting whenever the response variable is bounded and the
model $\mathcal{G}$ is convex and uniformly bounded. It can also be found in 
\cite{lecue:tel-00654100}, Proposition 1.3.3. 

\begin{proposition}
\label{prop_quad_curv}If the model $\mathcal{G}$ is convex and Conditions %
\ref{cond_reg_bounded} and \ref{cond_model_bounded} hold, then there exists
a constant $C>0$ such that%
\begin{equation}
P\left( f-f^{0}\right) \geq \frac{\sigma ^{2}\left( f-f^{0}\right) }{C^{2}}%
\text{ for all }f\in \mathcal{F}\text{ ,}  \label{ineq_margin_pointed}
\end{equation}%
Furthermore, $C=2\left( A_{1}+A_{2}\right) $ is convenient.
\end{proposition}

The major gain brought by Proposition \ref{prop_quad_curv} over the
classical margin relation is that the bias of the model, that is the
quantity $P\left( f^{0}-f_{\ast }\right) $ that is implicitly contained in
the excess risk appearing in the classical margin relation, is pushed away
from Inequality (\ref{ineq_margin_pointed}). Proposition \ref{prop_quad_curv}
is thus a refinement over the classical notion of margin relation. It is
stated for the least-squares contrast but it is easy to see that it can be
extended to more general situations, where the contrast $\gamma $ is convex
and regular in some sense (see Proposition 2.1, Section 2.2.3 in \cite%
{Saumard:10}). For completeness, the proof of Proposition can be found in
Section \ref{ssection_proof_framework}.

\subsection{Second order quadratic margin condition\label%
{ssection_second_order_margin}}

First notice that the arguments of the empirical process of interest can be
decomposed into a linear and a quadratic part. It holds, for any $f=f_{g}\in 
\mathcal{F}$ and any $\left( x,y\right) \in \mathcal{X\times }\mathbb{R}$,%
\begin{eqnarray}
f_{g}\left( x,y\right) -f^{0}\left( x,y\right) &=&\gamma \left( g\right)
\left( x,y\right) -\gamma \left( g^{0}\right) \left( x,y\right)  \notag \\
&=&\psi \left( x,y\right) \cdot \left( g-g^{0}\right) \left( x\right)
+\left( g-g^{0}\right) ^{2}\left( x\right) \text{ ,}
\label{contrast_expansion}
\end{eqnarray}%
where $\psi \left( x,y\right) =-2\left( y-g^{0}\left( x\right) \right) $.

To this contrast expansion around the projection $g^{0}$ of the regression
function onto $\mathcal{G}$, we can associate two empirical processes, that
we will call respectively the linear and the quadratic empirical process,
and we will be more precisely interested by their local maxima on $\mathcal{G%
}_{s}:=\left\{ g\in \mathcal{G}\text{ };\text{ }\left\Vert
g-g^{0}\right\Vert \leq s\right\} $, $s\geq 0$,%
\begin{equation*}
\mathbf{\hat{E}}_{n,\ell }\left( s\right) =\max_{g\in \mathcal{G}%
_{s}}\left\{ \left( P_{n}-P\right) \left( \psi \cdot \left( g-g^{0}\right)
\right) \right\} \text{ \ \ and \ \ }\mathbf{\hat{E}}_{n,q}\left( s\right)
=\max_{g\in \mathcal{G}_{s}}\left\{ \left( P_{n}-P\right) \left(
g-g^{0}\right) ^{2}\right\} \text{ .}
\end{equation*}%
In what follows, we will not directly show that the excess risk concentrates
around $s_{0}$ defined in (\ref{def_s0}), but rather around a point $\tilde{s%
}_{0}$, defined to be,%
\begin{equation*}
\tilde{s}_{0}=\arg \min_{s\geq 0}\left\{ s^{2}-\mathbf{E}_{\ell }\left(
s\right) \right\} ,\text{ }\mathbf{E}_{\ell }\left( s\right) =\mathbb{E}%
\left[ \mathbf{\hat{E}}_{n,\ell }\left( s\right) \right] \text{ .}
\end{equation*}%
It holds around $\tilde{s}_{0}$ a relation of the type of a second order
margin relation, introduced in \cite{vandeGeerWain:16}, as proved in the
following Lemma, which proof is available in Section \ref%
{ssection_proof_framework}. 

\begin{lemma}
\label{lemma_second_order_lin}For any $s\geq 0$, it holds%
\begin{equation}
s^{2}-\mathbf{E}_{\ell }\left( s\right) -\left[ \tilde{s}_{0}^{2}-\mathbf{E}%
_{\ell }\left( \tilde{s}_{0}\right) \right] \geq \left( s-\tilde{s}%
_{0}\right) ^{2}\text{ ,}  \label{second_order_lin}
\end{equation}%
and also%
\begin{equation}
s^{2}-\mathbf{\hat{E}}_{n,\ell }\left( s\right) -\left[ \tilde{s}_{0}^{2}-%
\mathbf{\hat{E}}_{n,\ell }\left( \tilde{s}_{0}\right) \right] \geq \left( s-%
\tilde{s}_{0}\right) ^{2}\text{ .}  \label{second_order_lin_emp}
\end{equation}
\end{lemma}

The fundamental difference with the second order margin relation stated in 
\cite{vandeGeerWain:16} is that we require in Lemma \ref%
{lemma_second_order_lin} conditions on the linear part of empirical process
and not on the empirical process of origin, that takes in arguments
contrasted functions. Indeed, it seems that for the latter empirical
process, the second order margin relation does not hold or is hard to check
in least-squares regression in general. This difficulty indeed forced van de
geer and Wainwright \cite{vandeGeerWain:16}, Section 6.2, to work in a
"linearized least-squares regression" context, under the quite severe
restriction that the distribution of design is known from the statistitian.
On contrary, our main result stated in Section \ref{section_main_result}\
below is stated for a general regression situation, where the distribution
of the design is unknown and the noise level is heteroscedastic.

\section{Main result\label{section_main_result}}

Before stating our new concentration inequality, we describe the required
assumptions. In order to state the next condition, let us denote 
\begin{equation*}
\mathbf{\hat{E}}_{n,1}\left( s\right) =\max_{g\in \mathcal{G}_{s}}\left\{
\left( P_{n}-P\right) \left( g-g^{0}\right) \right\} \text{ , \ }\mathbf{E}%
_{1}\left( s\right) =\mathbb{E}\left[ \mathbf{\hat{E}}_{n,1}\left( s\right) %
\right] \text{ .}
\end{equation*}

\begin{condition}
\label{cond_expected_bound}There is a sequence $m_{n}$ and a strictly
increasing function $\mathcal{J}_{1}$ such that the function 
\begin{equation*}
\Phi _{\mathcal{J}_{1}}\left( u\right) :=\left[ \mathcal{J}_{1}^{-1}\left(
u\right) \right] ^{2},\text{ \ }u>0,
\end{equation*}%
is strictly convex and such that%
\begin{equation*}
\mathbf{E}_{1}\left( s\right) \leq \frac{\mathcal{J}_{1}\left( s\right) }{%
m_{n}},\text{ \ }s\geq 0\text{.}
\end{equation*}
\end{condition}

\begin{condition}
\label{cond_envelope}Take $K>0$. For any $s>0$, there exists a constant $%
D\left( s\right) \in \left( 0,K\right] $ such that%
\begin{equation*}
\sup_{g\in \mathcal{G}_{s}}\left\Vert g-g^{0}\right\Vert _{\infty }\leq
D\left( s\right) \text{ .}
\end{equation*}
\end{condition}

Notice that if Conditions \ref{cond_reg_bounded}, \ref{cond_model_bounded}
and \ref{cond_expected_bound} hold, then by the use of classical
symmetrization and contraction arguments, for all $s\geq 0$,%
\begin{equation*}
\mathbf{E}_{\ell }\left( s\right) \leq C_{A_{1},A_{2}}\frac{\mathcal{J}%
_{1}\left( s\right) }{m_{n}}\text{ and }\mathbf{E}_{q}\left( s\right) \leq
4D\left( s\right) \frac{\mathcal{J}_{1}\left( s\right) }{m_{n}}\text{ ,}
\end{equation*}%
where $C_{A_{1},A_{2}}$ is a positive constant that only depends on $A_{1}$
and $A_{2}$. From now on, we set%
\begin{equation*}
\mathcal{J}:=\max \left\{ C_{A_{1},A_{2}};4\left( K\vee 1\right) \right\} 
\mathcal{J}_{1}\text{ ,}
\end{equation*}%
so that $\max \left\{ \mathbf{E}_{\ell }\left( s\right) ;\mathbf{E}%
_{q}\left( s\right) \right\} \leq \mathcal{J}\left( s\right) /m_{n}$ for any 
$s>0$ and also $\mathbf{E}_{q}\left( s\right) \leq D\left( s\right) \mathcal{%
J}\left( s\right) /m_{n}$.

We are now able to state our main result.

\begin{theorem}
\label{theorem_gene}If%
\begin{equation*}
\mathcal{J}\left( s\right) =A_{\mathcal{J}}s\text{ };\text{ }D\left(
s\right) =A_{\infty }s\text{ };\text{ }m_{n}=\sqrt{n}\text{ \ };\text{ \ }%
\tilde{s}_{0}\asymp \frac{A_{0}}{\sqrt{n}}
\end{equation*}%
then it holds, for any $t\geq 0$,%
\begin{equation}
\mathbb{P}\left( \left\vert \left\Vert \hat{g}-g^{0}\right\Vert -\tilde{s}%
_{0}\right\vert \geq \frac{\sqrt{A_{\mathcal{J}}A_{\infty }}\tilde{s}_{0}}{%
n^{1/4}}\vee c_{0}\left( \sqrt{\frac{t+\ln \left( 1+K\sqrt{n}\right) }{n}}+%
\frac{t+\ln \left( 1+K\sqrt{n}\right) }{n}\right) \right) \leq e^{-t}\text{ .%
}  \label{concen_ineq_gene}
\end{equation}%
Hence, if moreover $\sqrt{\ln n}\ll A_{0}\ll \sqrt{n}$ and $A_{\mathcal{J}%
}A_{\infty }\ll \sqrt{n}$ then,%
\begin{equation*}
\left\vert \frac{\left\Vert \hat{g}-g^{0}\right\Vert -\tilde{s}_{0}}{\tilde{s%
}_{0}}\right\vert =O_{\mathbb{P}}\left( \frac{\sqrt{A_{\mathcal{J}}A_{\infty
}}}{n^{1/4}}\vee \frac{\sqrt{\ln n}}{A_{0}}\right) =o_{\mathbb{P}}\left(
1\right) \text{ .}
\end{equation*}
\end{theorem}

Inequality (\ref{concen_ineq_gene}) of Theorem \ref{theorem_gene} is a new
concentration inequality related least-squares regression with random design
and heteroscedastic noise on a convex, uniformly bounded model. In
particular, it extends results of \cite{vandeGeerWain:16} related to
linearized regression, which is a simplified framework for regression, to
the classical and general least-squares framework described in Section \ref%
{ssection_setting}\ above. The proof of Theorem \ref{theorem_gene} is
detailed in Section \ref{ssection_proof_main_result}.

The following corollary provides an generic example entering into the
assumptions of Theorem \ref{theorem_gene}\ and that is related to linear
aggregation \textit{via} empirical risk minimization.

\begin{corollary}
\label{corollary_linear_model}Define $\mathcal{M}=%
%TCIMACRO{\TeXButton{\Span}{\Span}}%
%BeginExpansion
\Span%
%EndExpansion
\left( \varphi _{1},...,\varphi _{D}\right) $ the linear span generated by $%
\left( \varphi _{k}\right) _{k=1}^{D}$, an orthonormal dictionary in $%
L_{2}\left( P^{X}\right) $. Take $\mathcal{G}=B_{\left( \mathcal{M}%
,L_{\infty }\right) }\left( g^{0},1\right) \subset \mathcal{M}$ the unit
ball in sup-norm of $\mathcal{M}$ centered on $g^{0}$, the projection of $%
g_{\ast }$ onto $\mathcal{M}$. Assume that%
\begin{equation}
\sup_{g\in \mathcal{M},\left\Vert g\right\Vert _{2}=1}\left\Vert
g\right\Vert _{\infty }\leq c_{\mathcal{M}}\sqrt{D}  \label{hyp_unit_env}
\end{equation}%
and%
\begin{equation*}
\left( \ln n\right) ^{2}\leq D\leq \frac{\sqrt{n}}{\ln n}\text{ .}
\end{equation*}%
Then, if $\tilde{s}_{0}\asymp \sqrt{D/n}$, 
\begin{equation*}
\left\vert \frac{\left\Vert \hat{g}-g^{0}\right\Vert -\tilde{s}_{0}}{\tilde{s%
}_{0}}\right\vert =O_{\mathbb{P}}\left( \frac{\sqrt{D}}{n^{1/4}}\vee \sqrt{%
\frac{\ln n}{D}}\right) =o_{\mathbb{P}}\left( 1\right) \text{ .}
\end{equation*}
\end{corollary}

Note that Inequality (\ref{hyp_unit_env}) relating the sup-norm to the
quadratic norm of the functions in the linear span of the dictionary, is
classical in non-parametric estimation and is satisfied for the usual
functional bases such as the Fourier basis, wavelets or piecewise
polynomials over a regular partition (including histograms), see for
instance \cite{BarBirMassart:99}. In particular, Corollary \ref%
{corollary_linear_model} extends a concentration inequality recently
obtained in \cite{Saum:17}, for the excess risk of the ERM in the linear
aggregation problem when the dictionary at hand is the Fourier dictionary.

\section{Proofs\label{section_proofs}}

\subsection{Proofs related to Section \protect\ref{section_framework}\label%
{ssection_proof_framework}}

\begin{proof}[Proof of proposition \protect\ref{prop_quad_curv}]
Take $f=f_{g}\in \mathcal{F}$. Then, on the one hand,%
\begin{eqnarray}
\sigma ^{2}\left( f-f^{0}\right) &\leq &P\left( f-f^{0}\right) ^{2}  \notag
\\
&=&\mathbb{E}\left[ \left( \gamma \left( g\right) -\gamma \left(
g^{0}\right) \right) ^{2}\left( X,Y\right) \right]  \notag \\
&=&\mathbb{E}\left[ (g-g^{0})^{2}\left( X\right) (2Y-g\left( X\right)
-g_{0}\left( X\right) )^{2}\right]  \notag \\
&\leq &4\left( A_{1}+A_{2}\right) ^{2}\left\Vert g-g^{0}\right\Vert ^{2}%
\text{ .}  \label{majo_sigma}
\end{eqnarray}%
On the other hand, 
\begin{eqnarray}
P\left( f-f^{0}\right) &=&\mathbb{E}\left[ \left( \gamma \left( g\right)
-\gamma \left( g^{0}\right) \right) \left( X,Y\right) \right]  \notag \\
&=&\mathbb{E}\left[ (g^{0}-g)\left( X\right) (2Y-g\left( X\right)
-g_{0}\left( X\right) )\right]  \notag \\
&=&\mathbb{E}\left[ (g^{0}-g)\left( X\right) (2\left( Y-g^{0}\left( X\right)
\right) +g_{0}\left( X\right) -g\left( X\right) )\right]  \notag \\
&=&\left\Vert g-g^{0}\right\Vert ^{2}-2\mathbb{E}\left[ \left( Y-g^{0}\left(
X\right) \right) \left( g-g^{0}\right) \left( X\right) \right]  \notag \\
&=&\left\Vert g-g^{0}\right\Vert ^{2}-2\mathbb{E}\left[ \left( g_{\ast
}-g^{0}\right) \left( X\right) \left( g-g^{0}\right) \left( X\right) \right]
\notag \\
&\geq &\left\Vert g-g^{0}\right\Vert ^{2}\text{ .}  \label{mino_excess}
\end{eqnarray}%
The latter inequality, which corresponds to $\mathbb{E}\left[ \left( g_{\ast
}-g^{0}\right) \left( X\right) \left( g-g^{0}\right) \left( X\right) \right]
\leq 0$ comes from the fact that $\mathcal{G}$ is convex and so, $g^{0}$
being the projection of $g_{\ast }$ onto $\mathcal{G}$, the scalar product
in $L_{2}\left( P^{X}\right) $ between the functions $g_{\ast }-g^{0}$ and $%
g-g^{0}$ is nonpositive. Combining (\ref{majo_sigma}) and (\ref{mino_excess}%
) now gives the result.
\end{proof}

\bigskip

\begin{proof}[Proof of Lemma \protect\ref{lemma_second_order_lin}]
Inequality (\ref{lemma_second_order_lin}) derives from (\ref%
{second_order_lin_emp}) by taking expectation on both sides. Concerning the
proof of (\ref{lemma_second_order_lin}), it is easily seen that the function 
$s\mapsto \mathbf{\hat{E}}_{n,\ell }\left( s\right) $ is concave. Indeed,
take for $s_{i}\geq 0$, $i\in \left\{ 1,2\right\} $,%
\begin{equation*}
g_{n,s_{i}}=\arg \max_{g\in \mathcal{G}_{s_{i}}}\left\{ \left(
P_{n}-P\right) \left( \psi _{1}\cdot \left( g-g^{0}\right) \right) \right\} 
\text{ .}
\end{equation*}%
For any $\alpha ,\beta \in \left( 0,1\right) ,$ $\alpha +\beta =1$, if $%
g_{b}=\alpha g_{n,s_{1}}+\beta g_{n,s_{2}}$, then by the triangular
inequality, $g_{b}\in \mathcal{G}_{\alpha s_{1}+\beta s_{2}}$, which gives,%
\begin{equation*}
\mathbf{\hat{E}}_{n,\ell }\left( \alpha s_{1}+\beta s_{2}\right) \geq \left(
P_{n}-P\right) \left( \psi _{1}\cdot \left( g_{b}-g^{0}\right) \right)
=\alpha \mathbf{\hat{E}}_{n,\ell }\left( s_{1}\right) +\beta \mathbf{\hat{E}}%
_{n,\ell }\left( s_{2}\right) \text{ .}
\end{equation*}%
Now, from the concavity of $s\mapsto \mathbf{\hat{E}}_{n,\ell }\left(
s\right) $, we deduce that the function $s\mapsto s^{2}-\mathbf{\hat{E}}%
_{n,\ell }\left( s\right) $ is $1$-strongly convex, which implies (\ref%
{second_order_lin}).
\end{proof}

\bigskip

\subsection{Proofs related to Section \protect\ref{section_main_result}\label%
{ssection_proof_main_result}}

\begin{proof}[Proof of Theorem \protect\ref{theorem_gene}]
We prove the concentration of $\left\Vert \hat{g}-g^{0}\right\Vert $ at the
right of $\tilde{s}_{0}$ and arguments will be of the same type for the
deviations at the left. Take $t,\varepsilon >0,$ $J:=\left\lceil
K/\varepsilon \right\rceil $ and set, for any $j\in \left\{ 1,...,J\right\}
, $ the intervals%
\begin{equation*}
I_{j}:=\left( \left( j-1\right) \varepsilon +\delta +\tilde{s}_{0},\tilde{s}%
_{0}+\delta +j\varepsilon \right] \text{ .}
\end{equation*}%
We also set%
\begin{equation*}
z\left( t\right) =2C\tilde{s}_{0}\sqrt{\frac{t}{n}}+r_{0}\sqrt{\frac{t}{n}}+%
\frac{Kt}{n}\text{ .}
\end{equation*}%
It holds, for any $\delta >0$ (to be chosen later),%
\begin{eqnarray*}
&&\mathbb{P}\left( \tilde{s}_{0}+\delta <\hat{s}\leq K\right) \\
&\leq &\mathbb{P}\left( \exists s\in \left( \tilde{s}_{0}+\delta ,K\right) ,%
\text{ }s^{2}-\mathbf{\hat{E}}_{n}\left( s\right) \leq \tilde{s}_{0}^{2}-%
\mathbf{\hat{E}}_{n}\left( \tilde{s}_{0}\right) \right) \\
&\leq &\mathbb{P}\left( \exists s\in \left( \tilde{s}_{0}+\delta ,K\right) ,%
\text{ }s^{2}-\mathbf{\hat{E}}_{n}\left( s\right) \leq \tilde{s}_{0}^{2}-%
\mathbf{E}\left( \tilde{s}_{0}\right) +z\left( t\right) \right) +e^{-t}\text{
,}
\end{eqnarray*}%
where in the last inequality we used Lemma \ref{lemma_dev}. Furthermore, by
setting for all $j\in \left\{ 1,...,J\right\} $,%
\begin{equation*}
\mathbb{P}_{j}:=\mathbb{P}\left( \exists s\in I_{j},\text{ }s^{2}-\mathbf{%
\hat{E}}_{n}\left( s\right) \leq \tilde{s}_{0}^{2}-\mathbf{E}\left( \tilde{s}%
_{0}\right) +z\left( t\right) \right) \text{ ,}
\end{equation*}%
a union bound gives,%
\begin{equation*}
\mathbb{P}\left( \exists s\in \left( \tilde{s}_{0}+\delta ,K\right) ,\text{ }%
s^{2}-\mathbf{\hat{E}}_{n}\left( s\right) \leq \tilde{s}_{0}^{2}-\mathbf{E}%
\left( \tilde{s}_{0}\right) +z\left( t\right) \right) \leq \sum_{j=1}^{J}%
\mathbb{P}_{j}\text{ .}
\end{equation*}%
Now, for each index $j$ and for all $s\in I_{j}$, we have%
\begin{equation}
s^{2}-\mathbf{\hat{E}}_{n}\left( s\right) \geq \left( \left( j-1\right)
\varepsilon +\delta +\tilde{s}_{0}\right) ^{2}-\mathbf{\hat{E}}_{n}\left(
j\varepsilon +\delta +\tilde{s}_{0}\right) \text{ .}  \label{lower_s_Ij_0}
\end{equation}%
Furthermore, it holds for all $u>0$, with probability $1-e^{-u}$, 
\begin{eqnarray*}
\mathbf{\hat{E}}_{n}\left( j\varepsilon +\delta +\tilde{s}_{0}\right) &\leq &%
\mathbf{\mathbf{E}}\left( j\varepsilon +\delta +\tilde{s}_{0}\right)
+2C\left( \delta +j\varepsilon \right) \sqrt{\frac{u}{n}}+z\left( u\right) \\
&\leq &\mathbf{E}_{\ell }\left( j\varepsilon +\delta +\tilde{s}_{0}\right) +%
\mathbf{E}_{q}\left( j\varepsilon +\delta +\tilde{s}_{0}\right) \\
&&+2C\left( \delta +j\varepsilon \right) \sqrt{\frac{u}{n}}+z\left( u\right) 
\text{ ,}
\end{eqnarray*}%
where the first inequality comes from Lemma \ref{lemma_dev}. By Lemma \ref%
{lemma_second_order_lin}, we then have%
\begin{eqnarray*}
&&\left( \left( j-1\right) \varepsilon +\delta +\tilde{s}_{0}\right) ^{2}-%
\mathbf{E}_{\ell }\left( j\varepsilon +\delta +\tilde{s}_{0}\right) -\mathbf{%
E}_{q}\left( j\varepsilon +\delta +\tilde{s}_{0}\right) \\
&\geq &\tilde{s}_{0}^{2}-\mathbf{E}_{\ell }\left( \tilde{s}_{0}\right)
+\left( \delta +j\varepsilon \right) ^{2}-D\left( j\varepsilon +\delta +%
\tilde{s}_{0}\right) \mathcal{J}\left( j\varepsilon +\delta +\tilde{s}%
_{0}\right) /m_{n} \\
&&+\left( \left( j-1\right) \varepsilon +\delta +\tilde{s}_{0}\right)
^{2}-\left( j\varepsilon +\delta +\tilde{s}_{0}\right) ^{2} \\
&\geq &\tilde{s}_{0}^{2}-\mathbf{E}\left( \tilde{s}_{0}\right) +\left(
\delta +j\varepsilon \right) ^{2}-D\left( j\varepsilon +\delta +\tilde{s}%
_{0}\right) \mathcal{J}\left( j\varepsilon +\delta +\tilde{s}_{0}\right)
/m_{n} \\
&&-2\varepsilon \left( \tilde{s}_{0}+\delta +j\varepsilon \right)
+\varepsilon ^{2}\text{ .}
\end{eqnarray*}%
Putting the previous estimates in (\ref{lower_s_Ij_0}), we get, for all $%
s\in I_{j}$,%
\begin{eqnarray*}
s^{2}-\mathbf{\hat{E}}_{n}\left( s\right) &\geq &\tilde{s}_{0}^{2}-\mathbf{E}%
\left( \tilde{s}_{0}\right) +\left( \delta +j\varepsilon \right)
^{2}-D\left( j\varepsilon +\delta +\tilde{s}_{0}\right) \mathcal{J}\left(
j\varepsilon +\delta +\tilde{s}_{0}\right) /m_{n} \\
&&-2\left( C\sqrt{\frac{u}{n}}+\varepsilon \right) \left( \delta
+j\varepsilon \right) -2\varepsilon \tilde{s}_{0}+\varepsilon ^{2}-z\left(
u\right) \\
&\geq &\tilde{s}_{0}^{2}-\mathbf{E}\left( \tilde{s}_{0}\right) +\frac{1}{2}%
\left( \delta +j\varepsilon \right) ^{2}-D\left( j\varepsilon +\delta +%
\tilde{s}_{0}\right) \mathcal{J}\left( j\varepsilon +\delta +\tilde{s}%
_{0}\right) /m_{n} \\
&&-2\left( C\sqrt{\frac{u}{n}}+\varepsilon \right) ^{2}-2\varepsilon \tilde{s%
}_{0}+\varepsilon ^{2}-z\left( u\right) \text{ ,}
\end{eqnarray*}%
with probability $1-e^{-u}$. We require that 
\begin{equation*}
\left( \delta +j\varepsilon \right) ^{2}\geq 4D\left( j\varepsilon +\delta +%
\tilde{s}_{0}\right) \mathcal{J}\left( j\varepsilon +\delta +\tilde{s}%
_{0}\right) /m_{n}\text{ .}
\end{equation*}%
Using the assumptions, it is equivalent to require%
\begin{equation*}
\left( \delta +j\varepsilon \right) ^{2}\geq \frac{4\sqrt{A_{\mathcal{J}%
}A_{\infty }}\left( j\varepsilon +\delta +\tilde{s}_{0}\right) ^{2}}{\sqrt{n}%
}\text{ .}
\end{equation*}%
Whenever $A_{\mathcal{J}}A_{\infty }\leq \sqrt{n}$, the last display is true
if $\delta \sim A_{\mathcal{J}}A_{\infty }\tilde{s}_{0}/n^{1/4}$. We also
require that%
\begin{equation*}
\left( \delta +j\varepsilon \right) ^{2}\geq 4\left( 2\left( C\sqrt{\frac{u}{%
n}}+\varepsilon \right) ^{2}+2\varepsilon \tilde{s}_{0}-\varepsilon
^{2}+z\left( u\right) +z\left( t\right) \right) \text{ .}
\end{equation*}%
To finish the proof, we fix $\varepsilon =1/\sqrt{n}$ and $u=t+\ln \left( 1+K%
\sqrt{n}\right) $. In particular, $J\leq 1+K\sqrt{n}$ and 
\begin{equation*}
\sum_{j=1}^{J}\mathbb{P}_{j}\leq e^{-t}\text{ .}
\end{equation*}%
Our conditions on $\delta =\delta \left( t\right) $ become, for a constant $%
c_{0}$ only depending on $C$ and $K$,%
\begin{equation*}
\delta \left( t\right) \geq \frac{\sqrt{A_{\mathcal{J}}A_{\infty }}\tilde{s}%
_{0}}{n^{1/4}}\vee c_{0}\left( \sqrt{\frac{t+\ln \left( 1+K\sqrt{n}\right) }{%
n}}+\frac{t+\ln \left( 1+K\sqrt{n}\right) }{n}\right) \text{ .}
\end{equation*}%
since $\tilde{s}_{0}\vee r_{0}\ll 1$.
\end{proof}

\begin{proof}[Proof of Corollary \protect\ref{corollary_linear_model}]
Under Assumption (\ref{hyp_unit_env}), we have%
\begin{equation*}
\mathbf{E}_{1}\left( s\right) \leq \mathbb{E}^{1/2}\left[ \left( \max_{g\in 
\mathcal{G}_{s}}\left\{ \left( P_{n}-P\right) \left( g-g^{0}\right) \right\}
\right) ^{2}\right] \leq s\sqrt{\frac{D}{n}}\text{ ,}
\end{equation*}%
where we used two times Cauchy-Schwarz inequality. Hence, $m_{n}=\sqrt{n}$
and $\mathcal{J}\left( s\right) =s\sqrt{D}$ are convenient. Furthermore, by
Assumption (\ref{hyp_unit_env}), we have%
\begin{equation*}
\sup_{g\in \mathcal{G}_{s}}\left\Vert g-g^{0}\right\Vert _{\infty }\leq c_{%
\mathcal{M}}s\sqrt{D}\text{ .}
\end{equation*}%
We can thus apply Theorem \ref{theorem_gene} with $D\left( s\right) =c_{%
\mathcal{M}}s\sqrt{D}$. Consequently, condition $\sqrt{\ln n}\ll A_{0}\ll 
\sqrt{n}$ turns into 
\begin{equation*}
\ln n\ll D\ll n
\end{equation*}%
and condition $A_{\mathcal{J}}A_{\infty }\ll \sqrt{n}$ is satisfied whenever 
$D\ll \sqrt{n}$.
\end{proof}

In the following theorem (see Theorem 8.1 in \cite{vandeGeerWain:16}), the
right-tail inequalities are direct applications of Bousquet \cite%
{Bousquet:02} and the left-tail inequalities are deduced from Klein and Rio 
\cite{Klein_Rio:05}.

\begin{theorem}
\label{theorem_bou_klein_rio}If Conditions \ref{cond_reg_bounded} and \ref%
{cond_model_bounded} are satisfied and we set%
\begin{equation*}
K:=\sup_{f\in \mathcal{F}_{s}}\left\Vert f-f^{0}\right\Vert _{\infty }\text{
\ \ , \ }\sigma _{s}:=\sup_{f\in \mathcal{F}_{s}}\sigma \left(
f-f^{0}\right) \text{ ,}
\end{equation*}%
then it holds%
\begin{eqnarray*}
\mathbb{P}\left( \mathbf{\hat{E}}_{n}\left( s\right) \geq \mathbf{E}\left(
s\right) +\sqrt{\left( 8K\mathbf{E}\left( s\right) +2\sigma _{s}^{2}\right) 
\frac{t}{n}}+\frac{2Kt}{3n}\right) &\leq &e^{-t}\text{ ,} \\
\mathbb{P}\left( \mathbf{\hat{E}}_{n}\left( s\right) \leq \mathbf{E}\left(
s\right) -\sqrt{\left( 8K\mathbf{E}\left( s\right) +2\sigma _{s}^{2}\right) 
\frac{t}{n}}-\frac{Kt}{n}\right) &\leq &e^{-t}\text{ ,}
\end{eqnarray*}
\end{theorem}

Using Proposition \ref{prop_quad_curv} and Conditions \ref{cond_reg_bounded}%
, \ref{cond_model_bounded}, \ref{cond_expected_bound}, we can simplify the
bounds given in Theorem \ref{theorem_bou_klein_rio} as follows.

\begin{lemma}
\label{lemma_dev}If Conditions \ref{cond_reg_bounded}, \ref%
{cond_model_bounded} and \ref{cond_expected_bound} are satisfied, then with
the same notations as in Theorem \ref{theorem_bou_klein_rio} and also%
\begin{equation*}
r_{0}^{2}:=2C^{2}\Phi ^{\ast }\left( \frac{8K}{m_{n}C^{2}}\right)
\end{equation*}%
we have%
\begin{eqnarray*}
\mathbb{P}\left( \mathbf{\hat{E}}_{n}\left( s\right) \geq \mathbf{E}\left(
s\right) +2Cs\sqrt{\frac{t}{n}}+r_{0}\sqrt{\frac{t}{n}}+\frac{2Kt}{3n}%
\right) &\leq &e^{-t}\text{ ,} \\
\mathbb{P}\left( \mathbf{\hat{E}}_{n,\ell }\left( s\right) \leq \mathbf{E}%
\left( s\right) -2Cs\sqrt{\frac{t}{n}}-r_{0}\sqrt{\frac{t}{n}}-\frac{Kt}{n}%
\right) &\leq &e^{-t}\text{ ,}
\end{eqnarray*}
\end{lemma}

\begin{proof}
We have $\sigma _{s}^{2}\leq C^{2}s^{2}$, where the constant $C$ is defined
in proposition \ref{prop_quad_curv}. Furthermore, using the fact that $%
uv\leq \Phi _{\mathcal{J}}\left( u\right) +\Phi _{\mathcal{J}}^{\ast }\left(
v\right) $ for any $u,v>0$, we get%
\begin{equation*}
8K\mathbf{E}\left( s\right) \leq \frac{16K}{m_{n}}\mathcal{J}\left( s\right)
\leq 2C^{2}s^{2}+2C^{2}\Phi _{\mathcal{J}}^{\ast }\left( \frac{8K}{m_{n}C^{2}%
}\right) =2C^{2}s^{2}+r_{0}^{2}\text{ .}
\end{equation*}%
The conclusion is then easy to obtain by using $\sqrt{a+b}\leq \sqrt{a}+%
\sqrt{b}$.
\end{proof}

\bibliographystyle{alpha}
\bibliography{chern,Slope_heuristics_regression_13}

\end{document}